\documentclass[12pt,a4paper]{amsart}
\usepackage{amsfonts,amssymb}
\numberwithin{equation}{section}

\theoremstyle{plain}
\newtheorem{Th}{Theorem}[section]
\newtheorem{Lemma}[Th]{Lemma}
\newtheorem{Cor}[Th]{Corollary}
\newtheorem{St}[Th]{Statement}

 \theoremstyle{definition}
\newtheorem{Def}[Th]{Definition}
\newtheorem{Conj}[Th]{Conjecture}
\newtheorem{Rem}[Th]{Remark}
\newtheorem{Ex}[Th]{Example}

\newtheorem{?}[Th]{Problem}

\newtheorem{Pro}[Th]{Problem}
\newtheorem{Prob}[Th]{Problem}

 \newcommand{\zj}[1]{\left({#1}\right)}
 \newcommand{\low}[1]{\left\lfloor{#1}\right\rfloor}
 \newcommand{\up}[1]{\left\lceil {#1}\right\rceil }

     \newcommand\B{{\mathcal B}}

\newcommand{\setZ}{{\mathbb Z}}
\newcommand{\setR}{{\mathbb R}}
\newcommand{\setN}{{\mathbb N}}
\newcommand\1[1]{{ \bf 1}_{#1}}
\newcommand\Mf{{\overline M}}
\newcommand\Ma{{\underline M}}
\newcommand\Mabs{{M^*_*}}
\newcommand\df{{\overline d}}
\newcommand\da{{\underline d}}

\newcommand\coo{\overset{1}{\sim}}
\newcommand\col{\overset{{\rm lim}}{\sim}}
\newcommand\vr{\Subset}

\begin{document}

%\title{Density, mean, measure}
\title{Densitometria I. Discrete groups}

\author{Szil\' ard Gy. R\' ev\' esz}
\address{HUN-REN Alfr\' ed R\' enyi Institute of Mathematics\\
     Budapest, Pf. 127\\
     H-1364 Hungary}
\email{revesz.szilard@renyi.hu}
\thanks{Supported by NKFIH projects No. K-146387, K-147153 and E-151341.}

\author{Imre Z. Ruzsa}
\address{HUN-REN Alfr\' ed R\' enyi Institute of Mathematics\\
     Budapest, Pf. 127\\
     H-1364 Hungary}
\email{ruzsa@renyi.hu}
\thanks{Supported by NKFIH projects No. K-146387 and K-133819.}



    \begin{abstract}
An upper mean here is a subadditive functional $\overline M$   defined on bounded functions on a commutative group
which has, beside some natural requirements,  the property we call restricted additivity:
if $g(x)= f(x)+f(x+t)$, then $\overline M(g)= 2 \Mf(f)$.  This tries to grasp that it should not depend on
local properties. This naturally induces a lower mean, and when they coincide it is the mean. Restriction
to 0--1 valued functions (sets) is a density.

We answer the following questions:

      Given a functional defined on a subset of all functions, when is it a mean?

      Given a functional, which is a mean, how do we find the upper mean it came from? Is it unique?

     Given a function $f$, what are the possible values of $\overline M(f)$, for upper means $\overline M$?

In particular, we find the extremal means and give several expressions for it. We propose the names ``lowest and uppermost mean'' for them to replace the not really justified names ``lower and upper Banach mean and density''.

We also consider analogous questions for densities, with partial answers only.
    \end{abstract}

     \maketitle

     \section{Introduction}

     There are many concepts of density in use. For sets of integers we have the
\emph{asymptotic upper density}, defined as
\[ \overline d(A) = \limsup \frac{|A \cap [1,n]|} n, \]
     its symmetric variant, where the interval $[-n, n]$ is used, the logarithmic density and others;
each has the corresponding lower density, and when they coincide, they are called the density.
There is one, generally called the \emph{Banach density} defined as
        \[ d^*(A) = \lim_{n\to\infty} \frac{ \max |A \cap [t+1, t+n]|} n . \]
We shall propose the alternative name ``uppermost density'', since we
not find any connection with Banach. The first similarly defined concept we could locate is
from P\'olya \cite{polya29} (he says it ``deserves less attention'', see page 561), and the first real application by Kahane \cite{kahane54}, see also \cite{Kahane-AnnInstF}, page 54.

% \color{blue}
In the 60s the notion resurfaced in several works. Landau \cite{Landau} referred to lectures of Beurling -- related studies of Beurling himself was published only later in his collected works, see \cite{beurling} and \cite{beurlingB}. The notion was introduced in relation to packing in geometry by Groemer \cite{Groemer}. The name ``Banach density'' originated most probably from F\"urstenberg \cite{Furst}. The extension of the density to groups -- in particular when the group is non-metrisable -- required new constructions. This has been done in \cite{GrochKutySeip} and also in \cite{dissertation}. The extended notion has been successfully applied to questions of interpolation \cite{GrochKutySeip}, densities of packing \cite{LCATur}, \cite{elena-szilard}, and syndetic sets \cite{MathPann}.
% \color{black}

Since there is no natural analog of the intervals in other groups, or there are several
possible candidates (balls, cubes etc.), the extension of this concept may proceed in several ways.
They happen to be equivalent, and equivalence is proved on an ad hoc basis. Our aim is to
develop a theory of densities in groups, and point out the special position of the generalized Banach density among them.

        A set can be identified with a 0-1 valued function. We found it preferable to treat
general bounded functions instead, then specialize to sets. So first we consider the concept of a mean.

In this first part we restrict ourselves to commutative discrete groups. The starting point was sets of positive integers, which is only a semigroup; we plan to return to this and to the more general situation where the structure is a locally compact commutative group.

     \section{Upper and lower means}

     Let $G$ be an infinite commutative (discrete) group.

     \begin{Def} % \label{}
     A functional $\overline M$, which assigns a real number to every bounded real function
defined on $G$, is an
 \emph{upper mean}, if the following hold:

     (a) Norming: if $f$ is a constant function, $f\equiv c$, then $\overline M(f)=c$.

     (b) Monotonicity: if $f\leq g$, then $\overline M(f)\leq \Mf(g)$.

     (c) Homogeneity: if $f=cg$ with a constant $c \geq 0$, then $\overline M(f)=c \overline M(g)$

     (d) Restricted additivity: if $g(x)= f(x)+f(x+t)$,
 then $\overline M(g)= 2 \Mf(f)$.

     (e) Subadditivity: $\overline M(f+g) \leq  \overline M(f)+\overline M(g)$.
     \end{Def}

Translation invariance will follow from these assumptions. Property (d) is a way to express that
the men does not depend on local properties, like $x$ being odd or even. On the other hand,
invariance under automorphisms of  $G$ (e.g. $x  \hookrightarrow -x$ in $\setZ$)  is not assumed.
Upper asymptotic mean on $\setN$, that is
 \[ \limsup \frac{1}{n} \sum_{j=1}^n f(j) \]
is an admissible upper mean in our sense.

{\bf Convention. }
In the sequel we shall write simply ``function'' to mean  bounded real function
defined on $G$.

     \begin{St}
       Instead of assumption (b) above, the following special case suffices:

(b') Negativity:  if $f\leq 0$, then $\overline M(f)\leq 0$.
     \end{St}
     \begin{proof}
       Indeed, if $f\leq g$, by subaditivity we get
        \[ \Mf(f) \leq \Mf(g) + \Mf(f-g) \leq \Mf(g) .\]
     \end{proof}

     \begin{Rem}
       Instead of (b) it is not sufficient to assume

(b'') Positivity: if $f\geq  0$, then $\overline M(f)\geq 0$.
     \end{Rem}

     \begin{Ex}
       On the set of integers define
        \[ M(f) = \limsup \frac{1}{n} \zj{2 \sum_{i=1}^{2n} f(i) - 3\sum_{i=1}^{n} f(i) }. \]
Properties (a), (c), (e) are obvious. Property (d) is easy and we suppress the simple proof.
For a nonnegative $f$ let
 \[ \alpha = \limsup \frac{1}{n}\sum_{i=1}^{n} f(i),  \]
the usual asymptotic upper mean. For all large $n$ we have
 \[  \frac{1}{n}\sum_{i=1}^{n} f(i) < \alpha+\varepsilon,  \]
and for infinitely many values
\[  \frac{1}{n}\sum_{i=1}^{2n} f(i) > 2 \alpha-\varepsilon,  \]
hence
 \[ M(f) > 2(2\alpha-\varepsilon) - 3(\alpha+\varepsilon) = \alpha - 5\varepsilon,\]
which shows the positivity of $M(f)$.

 To see a negative
 function with a positive $M$ let
  \[ f(n) =
  \begin{cases}
    -1 & \text{ if } 2^{2i-1} <n \leq 2^{2i}, \\ 0 & \text{ otherwise.}
  \end{cases} \]
For $n=2^{2i}$ the sum in $(n+1, 2n]$ vanishes, hence we have
 \[  \frac{1}{n} \zj{2 \sum_{i=1}^{2n} f(i) - 3\sum_{i=1}^{n} f(i) }= - \frac{1}{n}  \sum_{i=1}^{n} f(i) > n/2,\]
so $M(f)\geq 1/2$.
     \end{Ex}

     \begin{St}
Instead of (d) we can assume

(d') \emph{restricted substractivity}: if $h(x)= f(x)-f(x+t)$,
then $\overline M(h)= 0$.
     \end{St}
     \begin{proof}
        Restricted substractivity implies resricted additivity as follows. If $g(x)= f(x)+f(x+t)$,
        then $g+h=2f$, so
         \[ 2 \Mf(f) = \Mf(2f) \leq \Mf(g) + \Mf(h) = \Mf(g), \]
         and
          \[ \Mf(g) \leq \Mf(2f) + \Mf(-h) = \Mf(2f); \]
we used that $-h(x) = f'(x)-f'(x-t)$ with $f'(x)=f(x+t)$.

The converse implication is a particular case of the following theorem.
     \end{proof}

\begin{Th}[General restricted additivity] \label{genrest}
Let  $\overline M$ be an upper mean, $f$ a  function, $c_i$ a sequence of real numbers such that
$ \sum |c_i| < \infty $ and $ \sum c_i \geq 0$, $t_i \in G$. For $g(x) = \sum c_i f(x+t_i)$ we have
  $\overline M(g)=\sum c_i \Mf(f)$.
\end{Th}
\begin{proof}
  We will now and then use the following easy consequence of the assumptions:
   \[ | \Mf(f) - \Mf(g) | \leq \sup |f(x)-g(x)| . \]
  Consider first the case when the sequence is finite, of length $k$, and $\sum c_i = 1$.

  We can write restricted additivity in the form
   \[ \Mf (f) = \Mf \zj{ \frac{f(x)+f(x+t)}{2} }\]
(we use $x$ as a dummy variable). Applying this equation to the right hand side with $2t, 4t, \ldots, 2^{n-1}t$
in the place of $t$ we obtain
 \[ \Mf (f) = \Mf \zj{ 2^{-n} \sum_{j=0}^{2^n-1} f(x+jt)} .   \]
By applying this successively for $t=t_1, \ldots, t_k$ we get
 \[ \Mf (f) = \Mf \zj{ 2^{-kn} \sum_{j_1=0}^{2^n-1} \ldots  \sum_{j_k=0}^{2^n-1}  f(x+j_1t_1 + \ldots + j_kt_k)} =\Mf (f').   \]
The same transformation applied to the function $g$ yields
\[ \Mf (g) = \Mf \zj{ 2^{-kn} \sum_{j_1=0}^{2^n-1} \ldots  \sum_{j_k=0}^{2^n-1}  g(x+j_1t_1 + \ldots + j_kt_k)} =   \]
 \[ = \Mf \zj{ 2^{-kn} \sum_{i=1}^k c_i  \sum_{j_1=0}^{2^n-1} \ldots  \sum_{j_k=0}^{2^n-1}  f(x+j_1t_1 + \ldots + j_kt_k+t_i)} = \Mf(g').  \]
The sums defining $f'$ and $g'$ are almost identical: both contain summands of the form
$ f(x+j_1t_1 + \ldots + j_kt_k)$ with $0\leq j_i \leq 2^n$, and of total weight 1 if all $j_i$ satisfy $0< j_i < 2^n$.
If some $j_i$ is 0 or $2^n$, then their weights may differ, the difference being at most $\sum |c_i| $. The number of such
$k$-tuples is $\leq (2^n+1)^k-(2^n-1)^k$.
Consequently
 \[ |  \Mf(f) - \Mf(g) | =  |  \Mf(f') - \Mf(g') | \leq  \sup |f'-g'| \]
 \[ \leq  2^{-kn} \zj{(2^n+1)^k-(2^n-1)^k} \sum |c_i| \sup |f(x)|. \]
This bound tends to 0 as $n \to\infty$ which concludes the proof of this subcase.

Homogenity yields the case of finitely many terms with $\sum c_i > 0$.

To show the general case take an $\varepsilon>0$, and choose $k$ so that
 \[ \sum_{i>k} |c_i| < \varepsilon . \]
 Consider
  \[ g'(x) = 2\varepsilon f(x) + \sum_{i=1}^k c_i f(x+t_i).    \]
  This is covered by the previous case, so we have
   \[  \Mf(g')=\zj{ 2\varepsilon + \sum_{i=1}^k c_i } \Mf(f) ,\]
   and also
    \[ | \Mf(g') - \Mf(g) | \leq \sup |g'(x)-g(x)| \leq 3\varepsilon \sup |f(x)| . \]
These equations together imply the claim.
\end{proof}

To show that a given functional is an upper mean we shall check criteria (a), (b'), (c), (d'), (e).

      \begin{Def} % \label{}
   A functional $\underline M$, which assigns a real number to every function
defined on $G$, is a
 \emph{lower mean}, if it satisfies requirements (a)-(d) above, and instead of (e) it has the
opposite property:

     (e') Superadditivity: $\underline M(f+g) \geq  \underline M(f)+\underline M(g)$.
     \end{Def}

     There is a natural correspondence between upper and lower means:
$\underline M$ is a lower mean if and only if
$\overline M(f)=-\underline M(-f)$ is an upper mean. We shall call such a pair of
an upper and a lower mean \emph{conjugates}. (Exercise: show that $\Ma(f) \leq \Mf(f)$.)

\section{The mean when it exists}

If $\overline M(f)=\underline M(f)$, we call their common value $M(f)$ the \emph{mean},
and the function \emph{medial}.

In the sequel we try to aswer the following questions

      Given a functional defined on a subset of all functions, when is it a mean?

      Given a functional, which is a mean, how do we find the upper mean it came from? Is it unique?

     Given a function $f$, what are the possible values of $\overline M(f)$, for upper means $\overline M$?

     \begin{Th}\label{medials}
      Medial functions  form a linear subspace $L$, and the mean is a linear functional on $L$.
It has the following further properties:

  (a) Norming: constants belong to $L$, and if $f$ is a constant function, $f\equiv c$, then $ M(f)=c$.

     (b) Monotonicity: if $f\leq g$, $f.g\in L$, then $ M(f)\leq  M(g)$.

     (c) Closedness: if a function $f$ has the property that for every $\varepsilon>0$ there are
 functions
$\underline g, \overline g \in L$ such that $\underline g\leq f \leq  \overline g$ and
$ M(\overline g)-M(\underline g) < \varepsilon$, then $f$ is also medial.

(d) For every function $f$, the functions $f(x+t)-f(x)$ are medial and their mean is 0.
     \end{Th}

     \begin{proof}
If $f$ is medial, then $ \Ma(f) = \Mf(f) = -\Ma(-f)$, hence homogeneity extends to
 negative numbers: $M(cf)=cM(f)$ for all real $c$.

 If $f,g$ are medial, then
  \[ \Mf(f+g) \leq \Mf(f) + \Mf(g) = \Ma(f) + \Ma(g) \leq \Ma(f+g) ,\]
hence mediality of the sum and additivity of the mean follow.

Properties (a)-(d) immediately follow from the definition.
     \end{proof}

$L$ is typically a small subspace (later we shall describe the smallest), but it can be the whole space.

\begin{Def}
  We say that a functional $M$, defined on a subspace $L$ of bounded functions is a \emph{mean},
if there is an upper mean $\Mf$ with conjucate lower mean $\Ma$ such that $\Ma(f)=\Mf(f)$ holds
exactly for $f\in L$ and then this value is $M(f)$.
\end{Def}

\begin{Def}
  We call a mean \emph{full}, if every bounded function is medial.
\end{Def}

\begin{Th} \label{full}
  Given an upper mean $\Mf$ with conjugate lower mean $\Ma$, there is a full mean $M_1$ satisfying $\Ma(f)\leq M_1(f) \leq \Mf (f)$ for all $f$.
(Consequently $M_1$ is an extension of $M$).
\end{Th}
\begin{proof}
  As $M \leq \Mf$ and $\Mf$ is subadditive, the Hahn-Banach theorem yields the existence of a linear extension $M_1$
of $M$ satisfying $M_1 \leq \Mf$. To see that it is a mean, that is, both and upper and lower mean, note that
homogeneity, sub- and superadditivity are included in linearity. Norming and restricted substractivity are
consequences of being an extension of $M$; for the latter note that all difference functions $f(x)-f(x+t)$
are medial with mean 0. To
show monotonicity, observe that if $f\leq g$, then
 \[ M_1(f)-M_1(g) = M_1(f-g) \leq \Mf(f-g) \leq 0 . \]
\end{proof}

\begin{Th}\label{meancriteria}
  Let $M$ be a linear functional, defined on a linear subspace $L$ of bounded functions. $M$ is a mean
if and only if it satisfies (a)-(d) of Theorem \ref{medials}.
\end{Th}

This will be proved together with the next theorem.

       Given a mean, the upper mean it came from is not unique, but we can find the largest upper mean
and the smallest lower mean among the possible choices.

 \begin{Def}
   Given a mean $M$ on a subspace $L$ of medial functions,, its \emph{upper and lower envelope} are
    \[ M^e(f) = \inf_{g\in L, g \geq f} M(g), \ M_e(f) = \sup_{g\in Lg\leq f} M(g) , \]
    over medial $g$.
 \end{Def}

 \begin{Th}
The upper and lower envelope of a mean $M$ are upper and lower means, resp., which induce the  mean $M$.
 \end{Th}
 \begin{proof}
We show that for any functional $M$ satisfying conditions (a)-(d) of Theorem \ref{meancriteria}
the functional $M^e$ defined above is an upper mean. The only nontrivial property is restricted additivity.
Let $g(x)=f(x)+f(x+t)$, and take a medial $h$ such that $g\leq h$, $M(h) < M^e(g) + \varepsilon$. Then
 \[ 2f(x) \leq h(x) + \bigl(f(x)-f(x-t) \bigr) ,\]
hence $2M^e(f) \leq M(h)  < M^e(g) + \varepsilon$. As we can do this for every positive $\varepsilon$, we conclude
$M^e(g) \geq 2 M^e(f)$. The converse inequality is obvious.

Clearly $M^e$ and $M_e$ both extend $M$. Conversely, if $M^e(f)=M_e(f)$, then $f$ is medial by property (c).
 \end{proof}

On the other hand, there is no smallest upper and largest lower mean for a given mean $M$. Indeed, if $M_1$ is a full
extension of $M$, then for any $c\in(0,1)$ the upper mean $\Mf_2 (f) = c\Mf (f) + (1-c) M_1(f)$ is smaller than $\Mf$.

\begin{Rem}
  The upper envelope of the asymptotic mean is not the asymptotic upper mean but
   \[  M^ef(f) = \lim_{\varepsilon\to0} \limsup_{x\to\infty} \frac{1}{\varepsilon x} \sum_{x<n<(1+\varepsilon)x} f(n) . \]
\end{Rem}

This was also observed by P\'olya\cite{polya29}.

\section{Commarginality}

We try to describe those classes of functions on which all means coincide.
To this end we introduce the concept of commarginality.

\begin{Def}
  Let $f,f'$ be bounded functions on $G$. We say that they are \emph{commarginal},
and write $f\sim f'$, if there
exists a bounded function $F$ on $G^2$ and a finite set $B\subset G$ such that $F(x,y)\neq0$ implies
$x-y\in B$, and we have always
        \[ \sum_y F(x,y)= f(x), \ \sum_x F(x,y)=f'(y). \]
        (Note that the previous condition implies that the sums are finite. Intuitive meaning:
cut each $f(x)$ into finitely many parts, move them by bounded distances and reassemble to get $f'$.)
\end{Def}

\begin{Th}
  Commarginality is an equivalence relation, and it is compatible with linear operations;
that is,

        (a) $f \sim f$,

        (b) if $f\sim f'$, then $f'\sim f$;

        (c) if $f\sim f'$ and $f'\sim f''$, then $f\sim f''$;

        (d) if $f\sim f'$, then $cf\sim cf'$ for every constant $c$;

        (e) if $f\sim f'$ and $g\sim g'$, then $f+g\sim f'+g'$.
\end{Th}
\begin{proof}
  All save transitivity (c) immediately follow from the definition. To prove transitivity assume first that $f'$
is bounded away from 0, say $|f'(x)| > c >0$ . The functions $f, f'$ are marginals of some $F$ restricted by
a finite set $B$; $f', f''$ of $F'$, $B'$. Now define
 \[ F"(x,z) = \sum_y \frac{ F(x,y) F'(y,z)}{f'(y)} . \]
This is a finite sum, it is bounded by our assumptions, and if the sum is nonempty, then we have
$x-y\in B$ and $y-z\in B'$ for some $y$, whence $x-z\in B''=B+B'$.
By summing first over $z$ we see that one marginal of $F"$ is indeed $f$, and by summing over $x$ we see that
the other one is $f"$.

To get rid of the assuption choose a number $c > \sup |f'|$, apply the previous case for the functions
$f+c, f'+c, f"+c$ and use linearity.
\end{proof}

\begin{Th}\label{commcomp}
        Commarginality is compatible with any upper or lower mean; that is, if
$f, f'$ are commarginal functions and  $\overline M$ is any upper mean, then
        $\overline M(f) =\overline M(f').$
\end{Th}

  \begin{Lemma}\label{aligno}
    For every finite $B\subset G$ and $\varepsilon >0$ there is a finite $C\subset G$ such that $ |B+C| / |C| < 1+\varepsilon$.
  \end{Lemma}

  Easy and well-known. We shall show a stronger result as Lemma \ref{joparketta}.

\begin{proof}[Proof of Theorem \ref{commcomp}]
  Let $F$ and $B$ be the function and the finite set inducing commarginality of $f$ and $f'$.
Assume first that $F>0$. Take an $\varepsilon>0$ and a finite $C\subset G$ such that $ |B+C| / |C| < 1+\varepsilon $.
Put
 \[ g(x) = \sum_{t\in C} f(x-t), \]

 \[ g'(x) = \sum_{t\in B+C} f'(x-t). \]
 We have
  \[ g(x) = \sum_{t\in C} \sum_y F(x-t,y) .\]
  All nonzero terms of this sum satisfy $x-t-y\in B$, hence $y\in x-t-B\subset x-C-B$, so
   \[ g(x) \leq \sum_{y\in x-C-B} \sum_z F(z,y) = g'(x) . \]
   Consequently
    \[ \Mf(g) = |C| \Mf(f) \leq \Mf(g') = |B+C| \Mf(f'),  \]
that is,
     \[  \Mf(f) \leq \frac{ |B+C|}{|C|} \Mf(f') < (1+\varepsilon) \Mf(f').    \]
Similarly we deduce $  \Mf(f') < (1+\varepsilon ) \Mf(f)$,
and as this holds for every positive $\varepsilon$, we infer $  \Mf(f')=  \Mf(f)$.

To get rid of the positivity assumption, we apply the positive case for the functions
 \[ F_1(x)  =        \begin{cases}
          F(x,y)+c & \text{ if } x-y\in B, \\
         0 & \text{ otherwise, }
        \end{cases}   \]
with some $c> - \inf F$, $f_1(x)=f(x)+c |B|$, $g_1(x)=g(x)+c |B|$.
\end{proof}

Some weaker relations still imply coincidence of means.

\begin{Def}
  Let $f,f'$ be bounded functions on $G$. We say that they are \emph{$l^1$-commarginal},
and write $f \coo f'$, if there
exists a bounded function $F$ on $G^2$ and a function $h\in l^1(G)$ such that $|F(x,y)| < h(x-y)$ for all $x,y\in G$ and
        \[ \sum_y F(x,y)= f(x), \ \sum_x F(x,y)=f'(y). \]
\end{Def}

\begin{Def}
  Let $f,f'$ be bounded functions on $G$. We say that they are \emph{limit-commarginal},
and write $f \col f'$, if there exist two sequences $g_n, g_n'$ of bounded functions such that
 $g_n \sim g_n'$, $g_n \to f$ and $g_n' \to f'$  uniformly.
\end{Def}

\begin{St}
  $f \sim f' \to  f \coo f' \to f \col f' $ .
\end{St}

\begin{proof}
  The first implication is obvious. For the second, given the functions $F$ and $h$, take finite sets $B_n$ such that
   \[ \sum_{x \notin B_n} h(x) < 1/n . \]
Put
\[ F_n(x)  =        \begin{cases}
          F(x,y) & \text{ if } x-y \in  B_n, \\
         0 & \text{ otherwise, }
        \end{cases}   \]
\[ \sum_y F_n(x,y)= g_n(x), \ \sum_x F_n(x,y)=g_n'(y). \]
Clearly $g_n \sim g_n'$ and
         \[ |f(x)-g_n(x)| < 1/n, \   |f'(x)-g_n'(x)| < 1/n . \]
\end{proof}

\begin{St}[Asymmetric form of limit-commarginality.]\label{asym}
  If $f \col f'$, there are functions $f_n$ such that  $f_n \sim f$, $f_n \to f'$ uniformly.
\end{St}
\begin{proof}
  If $g_n \sim g_n'$, $g_n \to f$ and $g_n' \to f'$, we can put $f_n = g_n'-g_n+f$.
\end{proof}

\begin{Th}\label{limcomm}
  Both $l^1$-commarginality and limit-commarginality are equivalence relations,  compatible with linear operations
and any upper or lower mean.
\end{Th}
\begin{proof} Transitivity of  $l^1$-commarginality can be proved like that of commarginality above. Transitivity of
 limit-commarginality is shown as follows. Assume  $f \col f' \col f"$. By Statement \ref{asym} there are sequences of
functions $f_n \to f$, $f_n \sim f'$ and $f_n" \to f"$, $f_n" \sim f'$; hence $f_n \sim f_n"$
 and these sequences are sufficient using
the original definition. Reflexivity symmetry and compatibility of linear opertions of both relations are obvious.

To  show compatibility of limit-commarginality with upper means take
 sequences $g_n, g_n'$ of bounded functions such that $g_n \sim g_n'$, $g_n  \to f$ and $g_n' \to  f'$
 uniformly. We have $\Mf(g_n)=\Mf(g_n')$,
  \[ | \Mf(g_n) - \Mf(f) | \leq \sup |g_n-f| \to 0, \  | \Mf(g_n') - \Mf(f') | \leq  \sup |g_n'-f'| \to 0 \]
and these relations together yield the conclusion.

The claim for $l^1$-compatibility follows as it is a stronger assumption.
\end{proof}

\begin{Th}\label{osszeno}
  The equality $\Mf (f)=\Mf(g)$ holds for all upper means if and only if $f \col g$.
\end{Th}

This will be proved in Section \ref{prescribed} as a part of Theorem \ref{osszeno2}.

        \section{Uppermost and lowest means}

     For a given function $f$ consider all convex combinations of translates, that is, all
 functions of the form
     \[ g(x)= \sum _{i=1}^n c_i f(x+t_i), \ c_i\geq 0, \ \sum  c_i=1 .   \]
     Put
     \[   M^*(f) = \inf _g \sup _x g(x),  \]
     \[   M_*(f) = \sup _g \inf _x g(x). \]

     \begin{Th} \label{lf}
     $M^*$ is an upper mean, and $M_*$ is the conjugate lower mean. From among all possible upper
 means $M^*$ is maximal, and from among all possible lower means
$M_*$ is minimal.
     \end{Th}
     \begin{proof}
Norming, homogeneity and subadditivity of $M^*$ is clear. For restricted additivity take
 $f'(x)=\bigl( f(x) + f(x+t) \bigr)/2$. As the collection of convex combinations of translates
of $f'$ is a subset of those of $f$, we immediately see that $M^*(f') \geq M^*(f)$, while the converse inequality
follows from subadditivity.

 From Theorem    \ref{genrest} we get that $\Mf(f)=\Mf(g) \leq \sup g$ and hence
$\Mf(f) \leq M^*(f)$ for any upper mean.

The claims for lower means follow by obvious calculations.
     \end{proof}

     This  motivates the following name.

     \begin{Def} % \label{}
     We call $M^*$ the \emph{uppermost mean}, and $M_*$ the \emph{lowest mean}.
     \end{Def}

     If the uppermost and lowest means  coincide (a rare thing, a sort of quasi-periodicity), their value is
the \emph{absolute mean}, denoted by $\Mabs(f)$. Such functions  have the same mean  under all circumstances; we call them
\emph{absolutely medial}. Functions commarginal to a constant have this property.

\begin{Th}\label{absolute}
  A function is absolutely medial if and  only if it is  limit-commarginal to a constant.
\end{Th}
\begin{proof}
  Sufficiency follows from Theorem \ref{limcomm}. To show necessity take a function with $\Mabs(f)=\alpha$ and a positive integer $k$.
  By definition we find nonnegative constants $c_i, c_i'$ and $t_i, t_i'\in G$ such that $\sum c_i = \sum c_i'=1$ and
   \[ g(x) = \sum _{i=1}^n c_i f(x+t_i) < \alpha+1/k, \  g'(x) =  \sum _{i=1}^{n'} c_i' f(x+t_i') > \alpha-1/k \]
   for all $x$. Put
    \[ h_k(x) =  \sum _{j=1}^n  \sum _{i=1}^{n'} c_j c_i' f(x+t_j+t_i') =
 \sum _{i=1}^n c_i g'(x+t_i)=  \sum _{i=1}^{n'} c_i' g(x+t_i')  .             \]
 The first expression shows that $h_k \sim f$, the second that $h_k(x)  > \alpha-1/k$ and the third that $h_k(x)  < \alpha+1/k$;
hence $h_k \to\alpha$ uniformly and $f \col \alpha$.
\end{proof}

 \section{Uppermost mean as average}

{\bf Convention. }
Typically a ``set'' is a finite nonempty subset of $G$; we shall write
 $B\vr A$ to mean $B\subset A$, $B \neq \emptyset$, $|B| <\infty$.

 \begin{Def}
     Given a  set $B \vr G$,  the \emph{$B$-average (conditional expectation)} of a function is
     \[     (f|B) (t) = \frac{1}{|B|} \sum _{b\in B} f(b+t) .  \]
 \end{Def}

 \begin{Def}
   The \emph{expansion pseudomean} of a function $f$ is
    \[ E(f) =  \inf_{T\vr  G} \sup_{B\vr  G} \frac{1}{|B+T|} \sum _{b\in B} f(b). \]
 \end{Def}

In the next section we will see why it is not really a mean.

     \begin{Th}  \label{umm} If $ f \geq 0$, then
     \[   M^*(f) =  E(f) = \inf_{B\vr  G}    \sup (f|B) .  \]
%     \[   M_*(f) = \sup_{B\vr G}   \inf (f|B). \]
     \end{Th}

     \begin{proof} Since $f|B$ is a particular case of the linear combinations used in the definition
of $M^*$,  we immediately see that $  M^*(f) \leq  \inf_{B\vr G}    \sup (f|B)$.

Next we show that
  \[ M_1 :=  \inf_{B\vr G}    \sup (f|B) \leq E(f) .  \]
  To see this, given any $T\vr G$ and $\varepsilon>0$, we need to find a $B\vr G$ such that
   \[  \sum _{b\in B} f(b) > (M_1-\varepsilon) |B+T|.   \]
Take $\delta$ such that $M_1/(1+\delta)^2>M_1-\varepsilon  $, then a set $C$ such that $|C+T| < (1+\delta) |C|$ (Lemma \ref{aligno}).
By definition, for some $t$ we haveg
 \[ (f|C)(t) \geq M_1/(1+\delta), \]
 hence with  $B=C+t$ we have
 \[  \sum _{b\in B} f(b) = |B| (f|C)(t) \geq M_1 |B|/ (1+\delta) \geq M_1 |C+T| /(1+\delta)^2 \geq
 (M_1-\varepsilon) |B+T| .\]

Finally we show that $E(f) \leq M^*(f)$.

Take $t_i \in G$ and constants  $ c_i\geq 0$, $ \sum  c_i=1 $
such that
\[ g(x) = \sum_{i=1}^n c_i f(x+t_i) \leq M^*(f) + \varepsilon . \]
Write $ \{ t_i \} =T$. Take any set $B\vr G$. Write $A=B-T$.
We have
\[  \sum_{b\in B} f(b) = \sum_{b\in B} \sum_{i=1}^n c_i f(b-t_i+t_i) = \sum_{i=1}^n c_i \sum_{b\in B} f\bigl((b-t_i)+t_i\bigr)  \]
 \[ \leq \sum_{i=1}^n c_i \sum_{a\in A} f(a+t_i) =  \sum_{a\in A} g(a) \leq |A| \bigl(  M^*(f) + \varepsilon \bigr),   \]
 consequently
  \[ E(f) \leq \sup_{B\vr G} \frac{1}{|B-T|} \sum _{b\in B} f(b)  \leq  \bigl(  M^*(f) + \varepsilon \bigr).   \]
Nonnegativity was only used in this step, and in the next section we show that this  inequality
fails for general functions. \end{proof}

  \begin{Th} %  \label{umm2}
 For evey  $f$ we have
     \[   M^*(f)  = \inf_{B\vr G}    \sup (f|B) ),  \]
     \[   M_*(f) = \sup_{B\vr G}   \inf (f|B). \]
     \end{Th}

     \begin{proof}

The case of nonnegative functions is included in the previous theorem.
For the general case observe that
for constants we have $  M^*(f+c) = M^*(f) + c $ and $ (f+c|B) (t) = (f|B) (t) +c$,
so with suitable $c$ it reduces to the nonnegative case.
     \end{proof}

We do not need all finite sets for the above equalities.

     \begin{Def}  \label{tanudef}
       We call a collection  $\B$ of nonempty subsets of the group     $G$ a
\emph{witness}, if for every function $f$ we have
     \[   M^*(f) = \inf _{B\in \B} \sup (f|B).  \]
     \end{Def}

     \begin{Th}  \label{tanufelt}
 A collection  $\B$ of nonempty subsets of the group  $G$ is a witness if and only if for
     every finite set $T\vr G$ we have
\begin{equation}\label{tanu}
        \inf _{B\in \B} \frac{|B+T|}{|B|} = 1 . \end{equation}
     \end{Th}

     \begin{proof} Sufficiency is shown like in the previous proof, where the only fact we used was
that the set of all subsets satisifies \eqref{tanu}.

 To show necessity, assume that  $\B$ fails to satisfy \eqref{tanu}. This means that there is a finite
  set $T\vr G$ such that $ |B+T|/ |B|>c>1$ for all ${B\in \B}$. We may suppose $0\in T$, since this property
 is translation invariant.

 For each ${B\in \B}$ take an element
 $t(B)\in G$ such that all sets $B+T+t(B)$ are disjoint. The existence of such $t(B)$ is easily shown
 by a transfinite induction, since the cardinality of the collection of finite subsets of $G$
 is the same as that of $G$.

 For every $x\in G$ select a $z(x)\in x-T$ so that $z(x)\in B+t(B)$ if
  \[ (B+t(B)) \cap (x-T) \neq \emptyset, \]
in other words, if $x\in B+t(B)+T$
for some  ${B\in \B}$. By the disjointness assumption there can be at most one such $B$. If there
is no such $B$, let $z(x)\in x-T$ arbitrary.

Put
 \[ f = \sum_{x\in G} \1 {z(x)} .\]
 This function is commarginal to the identically 1 function, hence $M^*(f)=1$. On the other hand
  \[ \sum_{x\in B+t(B)} f(x) = |B+T|, \]
consequently $\sup (f|B) > c>1$ for all  ${B\in \B}$.
     \end{proof}

     Sometimes we can replace infimum by a limit.

     \begin{Th}  \label{sorozattanu}
     Let $B_i$ be a sequence of finite subsets of $G$ such that for every
 $t\in G$ we have
\begin{equation}\label{elnyel}
        \frac{ |(B_i+t) \setminus  B_i|}{|B_i|} \to  0 .   \end{equation}
     Then
\begin{equation}\label{sorozaton}
       M^*(f) = \lim_{i\to\infty} \sup (f|B_i) .  \end{equation}

Condition \eqref{elnyel} is also necessary.

Such a sequence exists if and only if $G$ is countable.

     \end{Th}
     \begin{proof}
       As for every finite $T\vr G$ we have
        \[ |B_i+T| \leq |B_i| + \sum_{t\in T}   |(B_i+t) \setminus  B_i| ,  \]
from condition \eqref{elnyel} we infer that $ |B_i+T| / |B_i| \to1$, hence every infinite collection of
the sets $B_i$ is a witness. Hence
 \[   M^*(f) = \inf _{i\in I} \sup (f|B_i)  \]
for every infinite set $I$ of integers, which is clearly equivalent to \eqref{sorozaton}.

Conversely, if
       \eqref{elnyel} fails for some $t$, then \eqref{tanu} fails for $T=\{0,t \}$ and the collection
$ \{B_i \}$ is not a witness.

If $G$ is countable, list the elements of $G$ as
 \[ G = \{ 0, g_1, g_2, \ldots\} . \]
 Take a set $B_k$ such that
  \[ | B_k + \{0, g_1, \ldots, g_k \} | < (1+1/k) |B_k|. \]

On the other hand, if $G$ is uncountable, let $H$ be the subgroup generated by $\bigcup B_i$. This is a countable
subgroup. Take any $t\in G\setminus H$. The sets $B_i$ and $B_i+t$ are disjoint, so property \eqref{elnyel} cannot hold.
     \end{proof}

 \section{The expansion mean for general functions}

 Given a function $f$ we put
  \[ f^+(x) = \max \{ f(x), 0 \}, \]
  the positive part of $f$.
  \begin{Th}
    For every function we have
    \[ E(f) = \min \{ \sup f, M^*(f^+) \}. \]
  \end{Th}
  \begin{proof}
    Clearly $E(f) \leq \sup f$, and $E(f) \leq E(f^+) =  M^*(f^+)$ by Theorem \ref{umm}.
To prove equality we will find for every $\varepsilon>0$ and finite $T\vr G$ a set $B\vr G$ such that
 \[ \frac{1} {|B+T|} \sum _{b\in B} f(b) >\min \{ \sup f, M^*(f^+) \} - \varepsilon.  \]

{\bf Case 1. } $M^*(f^+)>0$.

 As $E(f^+) =  M^*(f^+)$, there is a set $B'$ such that
\[ \frac{1} {|B'+T|} \sum _{b\in B'} f(b) > M^*(f^+)  - \varepsilon.  \]
If $\varepsilon< M^*(f^+)$, the sum must be positive, so the set
 \[ B = \{ b\in B': f(b) > 0 \}\]
 is not empty. Clearly $|B+T| \leq |B'+T|$ and
  \[ \sum _{b\in B} f(b) \geq  \sum _{b\in B'} f(b), \]
so this $B$ suffices.

{\bf Case 2. } $M^*(f^+)=0$

Take a $b\in G$ such that $f(b)> \sup f - \varepsilon$. Put $B = \{b \}$. Then
\[ \frac{1} {|B+T|} \sum _{b\in B} f(b) = \frac{f(b)}{ |T| } \geq \min \{ f(b), 0 \}. \]
Observe the dichotomy: if $f(b)<0$, the minimum is attained for $|T| =1$,
and  if $f(b)>0$, we get the infimum for  $|T| \to\infty$. This explains the
strange behavior of $E(f)$.

  \end{proof}

 \section{Means with prescribed values}\label{prescribed}

\begin{Th}\label{meanextension}
  Let $\mu$ be a  functional, defined on a  subset $\Lambda$ of bounded functions.
Assume that constants belong to $\Lambda$, and if $f$ is a constant function, $f\equiv c$, then $ \mu(f)=c$.
Suppose that

(*) for every $f_1, \ldots, f_k\in \Lambda$, $t_1, .., t_k\in G$ and $c_1, \ldots, c_k\in \setR$ such that
 \[ c_1 f_1(x+t_1) + ... + c_k f_k(x+t_k) \geq 0 \]
we have $ \sum c_i \mu(f_i) \geq 0$.

Then $\mu$ can be extended to a full mean.
\end{Th}
\begin{Rem}
Condition (*) is clearly necessary for the existence of such an extension.

$\Lambda$ need not be a subspace. We do not need to assume that $\mu$ is bounded and linear,
this will follow from condition (*).
\end{Rem}

\begin{proof}
Consider all extensions of $\mu$ that satisfy (*). By Zorn's lemma there is a maximal among them. Let
$M$ be such a maximal one, defined on a set $L \supset\Lambda$. We show that $M$ and $L$ satisfy the conditions
of Theorem \ref{meancriteria}.

Observe first that $L$ is translation invariant. Indeed, if (*) holds for a collection of tunctions,
it automaticallyy holds for the collection of translations of these functions.

Let $L'$ be the subspace generated by $L$. Any $g\in L'$ is of the form
 \[ g =c_1 f_1 + ... + c_k f_k, \ f_i\in L . \]
This representation is typically not unique, but the value of $\sum c_i M(f_i)$ is independent of the representation.
Indeed, if $\sum c_if_i =\sum c_i'f_i' $, then $\sum c_if_i - \sum c_i'f_i' $ and $ \sum c_i'f_i'- \sum c_if_i   $ are both nonnegatve, whence
 $\sum c_iM(f_i) -\sum c_i'M(f_i') $ is both nonegative and nonpositive.
This argument shows that $M'(g) = \sum c_iM(f_i)$ is a well-defined linear extension of $M$ to $L'$;
by the maximality assumption $M'=M$, $L'=L$.

From the conditions of  Theorem \ref{meancriteria}, (a), norming, is among our assumptions.
(b'), negaivity, is the case $k=1$ of (*).

     (c) Closedness means that if a function $f$ has the property that for every $\varepsilon>0$ there are  functions
$\underline g, \overline g\in L$ such that $\underline g\leq f \leq  \overline g$ and
$ M(\overline g)-M(\underline g) < \varepsilon$, then $f\in L$. If this were not the case, we could extend $M$
to $L \cup \{ f\}$ by putting
 \[ M'(f) = \sup_{\underline g\leq f} M(\underline g) = \inf_{  \overline g \geq f} M(  \overline g).  \]
We need to check that (*) holds for this extension, that is, if
\begin{equation}\label{1}
 c_1 f(x+t_1) + ... + c_k f(x+t_k) + h(x) \geq 0 \end{equation}
with $h\in L$, then
we have $ (\sum c_i)M'(f) + M(h)  \geq 0$. Now \eqref{1} implies that
 \[ \sum c_i g_i(x+t_i) + h(x) \geq 0, \]
where $g_i=  \overline g$ if $c_i \geq 0$ and $g_i=\underline g$ if $c_i<0$. Applying (*) to this equation we obtain
 \[   (\sum c_i)M'(f) + M(h)  \geq -\varepsilon  \sum |c_i |, \]
and we are done.

Finally we show
(d), that is, for every function $g$, the function $f(x)=g(x+t)-g(x)$ is in $L$ and $M(f)=0$.
If this is not the case, we  extend $M$
to $L \cup \{ f\}$ by putting $M'(f)=0$.
We need to check that (*) holds for this extension, that is, if
\begin{equation}\label{2}
 c_1 f(x+t_1) + ... + c_k f(x+t_k) + h(x) \geq 0 \end{equation}
with $h\in L$, then
we have $ M(h)  \geq 0$. (In (*) we can have many wighted translations of functions from $L$, but
we can combine them into a single one by the previous arguments.)

We know that $f$ is absolutely medial with mean 0, hence so is $F(x)= \sum c_i f(x+t_i)$.In particular,
$M^*(F)=0$, which means that there are constants $a_1, \ldots, a_m \geq 0$ and $u_1, \ldots, u_m\in G$
such that $\sum a_i =1$ and $\sum a_i F(x+u_i) < \varepsilon$.

Summing the assumption $F+h\geq0$ with these weights and translations we obtain
 \[ \sum a_i h(x+u_i) \geq -\varepsilon, \]
that is, $ \sum a_i (h(x+u_i) +\varepsilon) \geq 0$. By (*) this implies $M(h)\geq -\varepsilon$. As we have this for all
positive $\varepsilon$, we conclude $M(h) \geq 0$.

 So by Theorem \ref{meancriteria} $M$ is a mean, which by Theorem \ref{full} can be extended to a full mean.
\end{proof}

\begin{Pro}\label{upperextension}
  Let $\mu$ be a  functional, defined on a  subset $\Lambda$ of bounded functions.
When can $\mu$  be extended to an upper mean?
\end{Pro}

Om view of Theorem \ref{genrest}
the following condition is necessary. If $f_1, \ldots, f_k,g\in \Lambda$, $t_1, .., t_k, u_1, \ldots, u_m\in G$,
$c_1, \ldots, c_k, d_1, \ldots, d_m >0$ and
\[ c_1 f_1(x+t_1) + ... + c_k f_k(x+t_k) \geq d_1 g(x+u_1) + \ldots + d_m g(x+u_m), \]
then $ \sum c_i \mu(f_i) \geq  \sum d_j \mu(g)$. We cannot decide whether it is sufficient.

     \begin{Th}\label{givenfull}
       Let $f$ be a bounded function, and $\alpha\in[M_*(f), M^*(f)]$ any number. There is a full mean $M$ such that
$M(f)=\alpha$.
     \end{Th}
     \begin{proof}
       We apply the previous theorem for $\Lambda=\{f, \text{constants} \}$. We have to check condition (*), that is, if
\begin{equation}\label{3}
 c_1 f(x+t_1) + ... + c_k f(x+t_k) + c \geq 0, \end{equation}
then $\alpha \sum c_i + c \geq 0$. To this end separate the positive and negative coefficients.
Say, assume that $c_1, \ldots c_m \geq 0$ and $c_{m+1}, ..., c_k <0$. Write \eqref{3} as
\begin{equation}\label{4}
 c_1 f(x+t_1) + ... + c_m f(x+t_m) + c \geq (-c_{m+1}) f(x+t_{m+1}) + \ldots  + (-c_k) f(x+t_k).  \end{equation}
By comparing the uppermost means of the sides of \eqref{4} we obtain
 $M^*(f) \sum c_i + c \geq 0$. By comparing the lowest means  we obtain $M_*(f) \sum c_i + c \geq 0$.
A suitable convex combination yields the desired conclusion $\alpha \sum c_i + c \geq 0$.
     \end{proof}

Finally we prove Theorem \ref{osszeno} in a slightly expanded form.

\begin{Th}\label{osszeno2}
  For a pair $f,g$ of functions the following are equivalent.

  (a) The equality $\Mf (f)=\Mf(g)$ holds for all upper means.

  (b) The equality $M (f)=M(g)$ holds for every full mean.

  (c) $f-g$ is absolutely medial and $\Mabs(f-g)=0$.

  (d) $f \col g$.
\end{Th}
\begin{proof}
  The equivalence of (c) and (d) is contained in Theorem \ref{absolute}.

  The implication (a)$\to$(b) is obvious.

(c)$\to$(a) is immediate as
   \[    \Mf (f)\leq \Mf(g) +  \Mf(f-g) \leq \Mf(g) +  M^*(f-g) = \Mf(g)\]
and similarly we get $\Mf(g) \leq  \Mf (f)$.

To show (b)$\to$(c), assume that (c) fails. Then we can find an
 \[ \alpha \in [ M_*(f-g), M^*(f-g)], \ \alpha \neq 0 .\]
By Theorem \ref{givenfull} there is a full mean $M$ such that $M(f-g)=\alpha\neq0$, that is,
$M(f) \neq M(g)$.
\end{proof}

     \section{Upper and lower densities }

     \begin{Def} % \label{}
     We call a function $\overline d$ (or $\underline d$), assigning a real number to every subset of $G$
an \emph{upper density} (\emph{lower density}), if there exists an upper mean $\Mf$ (a lower mean $\Ma$)
such that
 \[ \df( A) = \Mf (\1 A), \ \da( A) = \Ma (\1 A) \]
for all $A \subset G$.
     \end{Def}

    A conjugate pair of an upper and a lower mean induces a conjugate pair of an upper and a lower density.
These are connected by  $\overline d(A)=1-\underline d(G\setminus A)$.
If $ \df(A)=\da(A)$, it is \emph{the density} $d(A)$.

We consider the following questions.

       When is a function, defined on subsets of $G$, an upper density?

       Does an upper density uniquely determine the upper mean used in the definition?

     When is a function, defined on a family of subsets of $G$, a density?

      Given a density, how do we find the upper density it came from? Is it unique?

     Given a subset of $G$, what are the possible values of its (upper/lower) densities?

     \begin{Def}
     We say that two sets $A, A'$ are \emph{perturbations} of each other and write $A\simeq A'$, if there is a finite set $B$
and a 1-to-1 mapping $\varphi:A \to A'$ such that  $\varphi(x)-x\in B$ for all $a\in A$. \end{Def}

     \begin{St}
       Any upper density  $\overline d$ (with conjugate lower density $\da$) has the following properties:

       (a) Norming: $ \overline d(\emptyset)=0$, $ \overline d(G)=1$.

       (b) Monotonicity: if $A\subset B$, then  $\overline d(A) \leq \overline d(B)$.

       (c) Restricted additivities: if the sets $A+t_i$ are disjoint, then $\overline d(\bigcup_{i=1}^n A+t_i) = n\overline d(A)$,
      $\da(\bigcup_{i=1}^n A+t_i) = n\da(A)$.

       (d) Sub- and superadditivity: $\overline d(A\cup B) \leq \overline d(A)+   \overline d(B)$,
$\da(A\cup B) \geq \da(A)+\da(B)$ if $A\cap B=\emptyset$.

       (e) Perturbation: if  $ A'$ is a perturbation of $A$, then $\overline d(A)=\overline d(A')$.
     \end{St}
     \begin{Rem}
       In (c) and (d) above, the requirements for upper and lower densities are not equivalent (at least not obviously).
Those with lower density could be reformulated using the upper one, which would seem less natural.
     \end{Rem}
     \begin{proof}
       Properties (a)-(d) are obvious. For (e) observe that $A\simeq A'$ implies $\1 A \sim \1 A'$; we need just to define
$F(x,y)=1$ if $x=\varphi(y)$, 0 otherwise.
     \end{proof}

     \begin{Pro}
       If  $\1 A \sim \1{ A'}$, is $A'$ ``almost'' a perturbation of $A$?

       Let $G=\setZ$, and consider
        \[ F(x,y) =        \begin{cases}
          1 & \text{ if } x=y\geq 0, \\
         -1 & \text{ if } x=y+1\geq 1.
        \end{cases} \]
This shows that the sets $\{0 \}$ and $\emptyset$ are commarginal (and hence so are every pair of finite sets), while clearly not paerturbations
of each other. On the other hand, it is easy to see that for \emph{infinite} $A,A'\subset\setZ$ the following are equivalent:
$A\simeq A'$; $\1 A\sim\1 A'$;
 \[ | A' \cap [-n, n] | =  | A \cap [-n, n] | + O(1) . \]
We dot not have a general exact expression for the ``almost'' above.
     \end{Pro}

     \begin{?} % \label{}
       Are the above properties sufficient for a function to be an upper density?
     \end{?}

     \begin{Th}\label{meanfromdesity}
       Let $\Mf$ be an upper mean and $\df$ the corresponding upper density.

       (a) For every bounded function $f$ we have
        \[ \Mf(f) = \inf \zj{\alpha + \beta \df(A)} \]
over real numbers $\alpha,\beta$, $\beta>0$ and sets $A$ such that there is a function $g\geq f$, $g \sim \alpha+\beta \1 A$.

(b) If $f \geq 0$, we also have
 \[ \Mf(f) = \inf \zj{  \beta \df(A) } \]
over nonnegative numbers $\beta$ and sets $A$ such that there is a function $g\geq f$, $g \sim \beta \1 A$.

(c) If $0 \leq f \leq 1$, then
  \[ \Mf(f) = \inf  \df(A) \]
over  sets $A$ such that there is a function $g\geq f$, $g \sim  \1 A$.
     \end{Th}

     \begin{Def}
       A finite set $Y\vr G$ is a \emph{tile}, if there is a $Z\subset G$ such that $Y+Z=G$ and the sets
$Y+z, \ z\in Z$ are disjoint.
     \end{Def}

     \begin{Lemma}
       There are arbitrarily large finite tiles.
     \end{Lemma}

A stronger result will be shown as Lemma \ref{joparketta}.

     \begin{proof}[Proof of Theorem \ref{meanfromdesity}.]

We start with (c). Given a function $f$ and an $\varepsilon>0$, we construct a function $g$ and a set $A$
such that  $g\geq f$, $g \sim  \1 A$ and $ \df(A)< \Mf(f)+\varepsilon$.

Take a tile $Y$, $|Y|>1/\varepsilon$ with set of translations $Z$. For each $z\in Z$ take a set $A_z \vr Y+z$ with cardinality
 \[ h(z) = \up{\sum_{x\in Y+z} f(x)} \]
   and let $A = \bigcup_{z\in Z} A_z $. For $x\in Y+z$ let
    \[ g(x) = f(x) + \zj{  \up{\sum_{y\in Y+z} f(y)} - \sum_{y\in Y+z} f(y)} \Biggm/ |Y| .   \]
Clearly $f(x) \leq g(x) < f(x)+1/|Y| < f(x)+\varepsilon$, so $\Mf(g)<\Mf(f)+\varepsilon$. Also  $g \sim  \1 A$
(to see this, note that both are commarginal to the function assuming $h(z)$ for $z\in Z$ and 0 otherwise).
This shows $ \Mf(f) \leq \inf  \df(A)$, while  $ \Mf(f) \geq \inf  \df(A)$ is obvious.

(b) is obtained by applying (c) for a function $f/\beta$ with any $\beta> \sup f$.

(a) is obtained by applying (b) for a function $f-\alpha$ with $\alpha$ chosen to make this positive.
     \end{proof}

     \begin{Cor}
     An upper density uniquely determines the corresponding upper mean.
     \end{Cor}

  Given a density, we try to find a corresponding upper density.

 \begin{Def}
   Given a density $d$, its \emph{upper and lower envelope} are the upper/lower densities
    \[ d^e(A) =  M^e(\1 A), \ d_e(A)=M_e(\1 A)  , \]
    where $M$ is the mean inducing $d$.
 \end{Def}

These are the extremal ones among possible choices of the suitable upper/lower densities. The above definition is
not really satisfactory, as it does not given an expression in terms of values of $d$.

 \begin{?}[Most annoying problem]
      For which densities is it true that
\[ d^e(A) =  \inf_{B\supset A} d(B), \ d_e(A)=  \sup_{B\subset A} d(B)  , \]
where $B$ runs over sets having a density?
  \end{?}

\bigskip
True for asymptotic and uppermost density in $\setZ$; probably not true always.

Let $\Mf$ be an upper mean, $M$, $\df$, $d$ the induced mean, upper density, density.
\begin{Def}
  If $\Mf=M^e$ (equivalently, $\df=d^e$), we call the upper mean/density \emph{regular}.
\end{Def}

\begin{Ex}
 Asymptotic upper density is not regular, the uppermost density is regular.
\end{Ex}

Regularity (almost) means that for every set $A$ and $\varepsilon>0$ we can find another set $B\supset A$ with a density such that
$d(B)<\df(A)+\varepsilon$.

\begin{Def}
  An upper density is \emph{strongly regular}, if for every set $A$ we can find another set $B\supset A$ with a density such that
$d(B)=\df(A)$.
\end{Def}

\begin{Ex}
 The uppermost density is not strongly regular, the upper envelope of asymptotic density is strongly regular.

\end{Ex}

        \section{Uppermost and lowest densities}

     \begin{Def} % \label{}
       We call the upper density induced by the uppermost mean the \emph{uppermost density},
the lower density induced by the lowest mean the \emph{lowest density}, denoted by
     $d^*$, and $d_*$, resp.
     \end{Def}

  \begin{Th}  \label{umd}
     \[   d^*(A) = \inf _B \max_t \frac{|A \cap (B+t)|}{|B|}
= \inf_X \sup_B \frac{|A \cap B|}{|B+X|}, \]
     infimum, supremum are over finite nonempty $B,X\vr G$.
     \end{Th}

This is a special case of Theorem \ref{umm}.

     For this equality we do not need all finite sets. Recall from Section 4 that we called
 a collection  $\B$ of nonempty subsets of the group     $G$ a
\emph{witness}, if for every function $f$ we have
     \[   M^*(f) = \inf _{B\in \B} \sup (f|B).  \]

Theorem \ref{tanufelt} tells that a
  collection  $\B$ of nonempty subsets of the group  $G$ is a witness if and only if for
     every finite set $T\vr G$ we have
\begin{equation}\label{tanu2}
        \inf _{B\in \B} \frac{|B+T|}{|B|} = 1 . \end{equation}

     \begin{Th}  \label{umdwitness}
     Let $\B$ be a family of finite sets.  If $\B$ is a witness, then
     \[   d^*(A) = \inf _{B\in \B}   \max_t \frac{|A \cap (B+t)|}{|B|}  . \]
       The above condition is necessary (witnesses for sets and functions are the same).
\end{Th}

\begin{Lemma}\label{joparketta}
  For every finite $B\vr G$ and $\varepsilon>0$ there is a tile $Y$ such that
   \[ |B+Y|/|Y| < 1+\varepsilon .\]
\end{Lemma}
\begin{proof}
  Let $H$ be the subgroup generated by $B$. As a finitely generated commutative group, $H$ is isomorphic
to $H' = \setZ^d \times F$ with some finite group $F$. Let $\varphi: H \to H'$ be the isomorphism. First we tile $H'$.
If $d=0$, put $Y'=H'$, $Z'=\{0 \}$. In this case
 \[ \frac{|\varphi(B)+Y'|}{|Y'|} =1. \]

 If $d \geq 1$, put $Y'=\{1,2, \ldots, n \}^d \times F$, $Z' = (n\setZ)^d \times \{0 \}$
with some $n$. If $m$ is the maximum of the absolute value of coordinates of $\varphi(B)$, then
 \[ \frac{|\varphi(B)+Y'|}{|Y'|} \leq \zj{\frac{n+2m}{n}}^d < 1+\varepsilon \]
if $n$ is large enough.

$Y= \varphi^{-1}(Y')$ satisfies $ |B+Y|/|Y| < 1+\varepsilon$ and tiles $H$ with $Z^*= \varphi^{-1}(Z')$.

Finally take a set $R$ which contains exactly one elment from each coset $H+t$ of $H$ and put
$Z=Z^*+R$ to see that $Y$ also tiles $G$.
\end{proof}

\begin{proof}[Proof of Theorem \ref{umdwitness}]
The ``if'' part of Theorem \ref{umdwitness} is a special case of Theorem
\ref{tanufelt}.

To prove necessity we shall construct, given a collection $\B$ that fails to satisfy \eqref{tanu2},
a set $A$ such that
 \[   d^*(A) < \inf _{B\in \B}   \max_t \frac{|A \cap (B+t)|}{|B|}  . \]
The fact
  that  $\B$ fails to satisfy \eqref{tanu2} means that there is a finite
  set $T\vr G$ such that $ |B+T|/ |B|>c>1$ for all ${B\in \B}$. We may suppose $0\in T$, since this property
 is translation invariant.

 Take a tile $Y$, tiling $G$ with some set $Z$, such that
  \[ \frac{|T+Y|}{|Y|} =c' <c .\]

 For each ${B\in \B}$ take an element
 $t(B)\in G$ such that all sets $B+T+t(B)$ are disjoint. The existence of such $t(B)$ is easily shown
 by a transfinite induction, since the cardinality of the collection of finite subsets of $G$
 is the same as that of $G$.

 For every $z\in Z$ select an $a(z)\in (Y+z) \cap (B+t(B))$ if there is such a  ${B\in \B}$.
 By the disjointness assumption there can be at most one such $B$. If there is no such $B$,
 let $a(z)\in Y+z$ arbitrary. Put
  \[ A = \{a(z): z\in Z . \} \]
Clearly $d^*(A)=1/|Y|$.

Now we estimate $n= |A\cap (B+t(B)|$. The set $B+t(B)$ is covered by $n$ translations of $Y$, hence
$(B+t(B))+T$ is covered by $n$ translations of $Y+T$. Comparing the sizes of these sets we obtain
 \[ c|B| < |B+T| \leq n|Y+T| = c' n |Y|, \]
 that is,
  \[ \frac{ |A\cap (B+t(B)|}{|B|} = \frac{n}{|B|} > \frac{c}{c'} \frac{1}{|Y|} =\frac{c}{c'} d^*(A).\]
\end{proof}

\bigskip

     Sometimes we can replace the infimum by a limit.

     \begin{Th} % \label{}
     Assume that $G$ is countable, and let $B_i$ be a sequence of finite subsets of $G$ with the property that for every
$x\in G$ we have
     \[   
     \frac{ |(B_i+x) \setminus  B_i| }{|B_i|} \to  0 .  
     \]
     Then
     \[   d^*(A) = \lim_{i\to\infty}   \max_t \frac{|A \cap (B_i+t)|}{|B_i|}.  \]
   \end{Th}

This is a special case of Theorem \ref{sorozattanu}.
   \begin{Rem}
       This is the way upper Banach density is usually defined in $\setZ$ (with intervals) or in $\setZ^k$ (with cubes or balls).
     \end{Rem}

\section{Density,  packing  and covering}

\begin{Def}
  Let $A\subset G$. For a positive integer $k$ let $p(k)$ (packing) denote the largest integer with the property that
there exist $t_1, \dots , t_k\in G$, such that every element of $G$ is contained in at most in $p(k)$ of the sets $A+t_i$.
\end{Def}

     \begin{Th}  \label{L}

 We have
     \[   \lim p(k)/k = d^*(A).  \]
     \end{Th}
     \begin{proof}
Observe first that $p$ is subadditive, hence the limit exists.

       Write $f=\1A$ and $g(x) = \sum f(x-t_j)$. We have $g(x) \leq p(k)$ and
        \[ k d^*(A)= k M^*(f) = M^*(g) \leq p(k), \]
hence $d^*(A) \leq p(k)/k$ for all $k$.

On the other hand, by Theorem \ref{umd} for every $\varepsilon>0$ we find a finite set $B$ such that
 \[ |A\cap(B+t)| \leq |B| (d^*(A)+\varepsilon) \]
 for all $t$. With $k= |B|$ and $-B= \{ t_1, \dots , t_k   \}$ this shows $p(k) \leq k( d^*(A)+\varepsilon)$, that is,
  \[ \limsup p(k)/k \leq d^*(A) . \]
     \end{proof}

     \begin{Def}
     Let $A\subset G$. For a positive integer $k$ let $lc(k)$ (covering) denote the largest integer with the property that
there exist $t_1, \dots , t_k\in G$, such that every element of $G$ is contained in at least in $c(k)$ of the sets $A+t_i$.
\end{Def}

\begin{Th}  \label{l}
 We have
     \[   \lim c(k)/k = d_* (A).  \]
     \end{Th}

     \begin{proof}
       Apply the previous theorem for the complement of $A$.
     \end{proof}

The following result was obsrved in several situations.

\begin{Th}
  Let $A$ be a set of positive uppermost density. We have
   \[ \up{\frac{1}{d_*(A-A)} } \leq \low{\frac{1}{d^*(A)} } ,     \]
consequently
 \[ d_*(A-A) \geq d^*(A) .\]
\end{Th}

\begin{proof}
  Let $k$ be the largest integer with the property that
there exist $t_1, \dots , t_k\in G$, such that  the sets $A+t_i$ are disjoint. In the notation of Theorem \ref{L}
this means that $p(k)=1$ and hence $k \leq 1/d^*(A)$. For every $x\in G$ the set $A+x$ must intersect some $A+t_i$,
that is, $x+a=t_i+a'$ with $a, a'\in A$ which means $x\in t_i + (A-A)$. This shows $d_*(A-A) \geq 1/k$
\end{proof}

\section{Density and measure}

Some densities, like the asymptotic and logarithmic ones, are defined as a limit of measures.
It would be interesting to explore the extent of this possibility.

\begin{Def}
  An upper density $\df$ is \emph{mensural}, if there is a sequence $\mu_j$ of measures on $G$ such that for all sets
  we have
   \[ \df (A) = \limsup \mu_j(A) .\]
\end{Def}

\begin{Prob}
  Which sequences of measures define a density in our sense?
\end{Prob}

\begin{Conj}
  Those which satisfy
   \[ \lim_{n\to\infty} \sum_{x\in G } |\mu_n(x)-\mu_n(x+t)| = 0\]
for all $t\in G$.
\end{Conj}

This condition is easily seen to be sufficient.

\begin{Prob}
  Find an inner characterization of mensural densities.
\end{Prob}

No conjecture. We can show that the uppermost density is not mensural.

% \bibliographystyle{amsplain}
% \bibliography{cimek}

\end{document}